\documentclass[11pt,a4paper,reqno]{amsart}
\usepackage[english]{babel}
\usepackage[applemac]{inputenc}
\usepackage[T1]{fontenc}
\usepackage{palatino}
\usepackage{bbm}
\usepackage{cite}
\usepackage{float, verbatim}
\usepackage{amsmath}
\usepackage{amssymb}
\usepackage{amsthm}
\usepackage{amsfonts}
\usepackage{graphicx}
\usepackage{mathtools}
\usepackage{enumitem}
\usepackage[colorlinks = true, citecolor = black]{hyperref}
\usepackage{color}
\usepackage{tikz}
\usepackage{xcolor}
\usepackage{pgfplots}
\usepackage{caption}
\usepackage{subcaption}
\usepackage{enumerate}
\usepackage{url}

\newcommand{\Haus}{\dim_{\mathrm{H}}}

\newtheorem*{thm*}{Theorem}

\newtheorem{thm}{Theorem}[section]

\newtheorem{lma}[thm]{Lemma}
\newtheorem{cor}[thm]{Corollary}
\newtheorem{defn}[thm]{Definition}

\newtheorem{conj}[thm]{Conjecture}
\newtheorem{rem}[thm]{Remark}
\newtheorem{ques}[thm]{Question}

\begin{document}
	
	\title{Times two, three, five orbits on $\mathbb{T}^2$}
	
	\author{Han Yu}
	\address{Han Yu\\
		Department of Pure Mathematics and Mathematical Statistics\\University of Cambridge\\CB3 0WB \\ UK }
	\curraddr{}
	\email{hy351@cam.ac.uk}
	\thanks{ }
	
	\subjclass[2010]{37C85, 57M60, 37A45 }
	
	\keywords{dynamical system, diagonal actions, torus automorphism}
	
	\date{}
	
	\dedicatory{}
	
	\begin{abstract}
		In this paper, we study orbit closures under diagonal torus actions. We show that if $(x,y)\in\mathbb{T}^2$ is not contained in any rational lines, then its orbit under the $\times 2, \times 3, \times 5$ actions is dense in $\mathbb{T}^2.$ 
	\end{abstract}
	\maketitle
	%\tableofcontents
	\allowdisplaybreaks
	\section{Introduction and backgrounds}
	\subsection{Backgrounds}
	In this paper, we are interested in subsets of $\mathbb{T}^d,d\geq 2$ which are simultaneously invariant under multiple torus endomorphisms. One of the first results in this area is due to Furstenberg, see \cite[PART IV]{Fu1}.
	
	\begin{thm}[Furstenberg]
		Let $A\subset\mathbb{T}$ be a closed $\times 2, \times 3$ invariant set. Then either $A$ is a finite set of rational numbers or $A=\mathbb{T}.$
	\end{thm}
	\begin{rem}
		This result also holds if we replace $2,3$ by another pair of integers $p,q>1$ such that $\log p/\log q\notin \mathbb{Q}.$
	\end{rem}
	Later on, this result was extended to deal with higher dimensional torus endomorphisms, see \cite{B84}. Let $d\geq 2$ and $\Sigma$ be a commutative  sub semigroup of $GL_d(\mathbb{Z}).$ Under an irreducibility condition on $\Sigma$, the only infinite closed $\Sigma$-invariant subset of $\mathbb{T}^d$ is $\mathbb{T}^d$ itself. The precise condition for $\Sigma$ is technical. Essentially, the most important part is that elements in $\Sigma$ do not have any common non-trivial invariant subspaces (subtori), see \cite{B84}, \cite{LW12} for more detailed discussions.
	
	One might be interested in dropping the irreducibility condition, which played a very central role in the proofs in \cite{B84}. However, without this condition, one cannot hope to obtain a result mentioned above. For example, consider $\times 2, \times 3$ actions on $\mathbb{T}^2.$ Clearly, all homogeneous lines (for example, the diagonal line) are invariant under both $\times 2, \times 3.$ It is natural to guess that essentially all closed invariant sets are special in some sense. In \cite{MP99} the following result was proved.
	
	\begin{thm}[Meiri and Peres]
		Let $d\geq 2$ be an integer. Let $A\subset\mathbb{T}^d$ be a closed $\times 2, \times 3$ invariant set. Then $\Haus A\in \{0,1,\dots,d\}.$
	\end{thm} 
	
	The result in \cite{MP99} is more general than what we stated here, a measure-theoretic version of the above result was also presented. Here $\Haus A$ is the Hausdorff dimension of $A,$ see \cite{Fa} for a systematic introduction. As a rational point has dimension zero, a line has dimension one, and $\mathbb{T}^2$ has dimension two, we see that the above result gives a coarse classification of closed invariant sets. In $\mathbb{T}^2$, it is easy to see that if $\Haus A=0$ then $A$ is finite and if $\Haus A=2$ then $A=\mathbb{T}^2.$ The intriguing case is when $\Haus A=1.$ We know that $A$ contains at most finitely many lines. However, $A$ could be a complicated union of lines and rational points. To have a clear view, we consider orbit closures. That is, $A$ is the topological closure of an orbit under $\times 2, \times 3.$ In this case, it might be the case that the following conjecture holds.
	\begin{conj}\label{conjecture}
		Let $d\geq 2$ be an integer. Let $A\subset\mathbb{T}^d$ be an orbit closure under $\times 2, \times 3$ actions. Then $A$ is a union of finitely many (possibly non-homogeneous) subtori.
	\end{conj}
	A homogeneous subtorus is a linear subspace of $\mathbb{T}^d$. A non-homogeneous subtorus is a translated copy of a homogeneous subtorus. In the setting of toral automorphisms, a related result was proved in \cite[Theorem 1.5]{LW12}.
	\begin{thm}
		Let $L_1,L_2,L_3$ be three automorphisms on $\mathbb{T}^d$ with at least one of the actions being totally irreducible. Consider the diagonal actions $L^{\Delta}_i,i\in\{1,2,3\}$ on $\mathbb{T}^d\times \mathbb{T}^d$ by
		\[
		L^{\Delta}_i(x,y)=(L_i(x),L_i(y)).
		\]
		Then any closed $L^{\Delta}_1,L^{\Delta}_2,L^{\Delta}_3$ invariant set is a union of finitely many (possibly non-homogeneous) subtori.
	\end{thm}
	\subsection{Main results in this paper}
	For clearness, given an integer $k\geq 1$, the $\times k$ map acts on $\mathbb{T}^2$ by multiplying $k$ on all coordinates. We write it as $T_k$ for short.  
	\subsection{$\times 2, \times 3, \times 5$ invariant subsets of $\mathbb{T}^2$}
	We start  with the following question.
	\begin{ques}\label{Ques}
		Let $(x,y)\in\mathbb{T}^2$ be an arbitrary point. Construct the orbit:
		\[
		Orb_{2,3,5}(x,y)=Orb_{T_2,T_3,T_5}(x,y)=\{T^{k_2}_2T^{k_3}_3T^{k_5}_5(x,y)\}_{k_2,k_3,k_5\in\mathbb{N}}.
		\]
		What can we say about $Orb_{2,3,5}(x,y)$?
	\end{ques}
	
	There are some simple cases to point out. Let $(x,y)$ be a rational point (i.e. $x,y\in\mathbb{Q}$). Then $Orb_{2,3,5}(x,y)$ is a finite set of rational points. The case when precisely one of $x,y$ is rational is similar. In this case, the orbit is dense in a set of lines. Suppose now that $x,y$ are not rational but $x,y$ are rationally related (i.e. we can find integers $m,n,k$ such that $(m,n,k)=1$ and $mx+ny=k$). In this case we see that $mx'+ny'\in\mathbb{Z}$ holds for each $(x',y')\in Orb_{2,3,5}(x,y).$ Since $(x,y)$ is not a rational point, we can use Furstenberg's theorem to conclude that $Orb_{2,3,5}(x,y)$ is dense in a union of lines. The last case is when $x,y$ are not rational nor rationally related. One can guess that $Orb_{2,3,5}(x,y)$ is dense in $\mathbb{T}^2.$ We will prove this guess.
	\begin{thm}\label{MAIN}
		Let $(x,y)\in\mathbb{T}^2$ be such that $x,y$ are not rational nor rationally related. Then $Orb_{2,3,5}(x,y)$ is dense in $\mathbb{T}^2.$ The converse is also true.
	\end{thm}
	We will prove this theorem in Section \ref{track}. Several remarks are in order.
	\begin{rem}
		The authors of [LW12] have announced that their methods can be extended to cover Theorem \ref{MAIN},
		the details of which are to be published in a forthcoming paper. See [LW12, Page 1308, last paragraph]. However, our proof in this paper is based on an entirely different approach.
	\end{rem}
	\begin{rem}
		We note that when $T_2,T_3,T_5$ are replaced by two actions $A_1,A_2$ which are irreducible (without common invariant non-trivial subspaces) then the above theorem is turned into a result in \cite{B84}. A key point here is to overcome the difficulty caused by the fact that $T_2,T_3,T_5$ are all diagonal actions carrying many common invariant non-trivial subspaces.
	\end{rem}
	\begin{rem}\label{23}
		It is perhaps strange that one needs three actions instead of two in Theorem \ref{MAIN}. We do not think that all three actions are necessary. However, at this stage, we are not able to tell whether or not $T_2,T_3$ are already sufficient to obtain a dense orbit in $\mathbb{T}^2$ in the statement of Theorem \ref{MAIN}.
	\end{rem}
	\begin{rem}
		One may ask whether this result can be proved on $\mathbb{T}^d, d\geq 3.$ That is, whether it is true or not that for $x\in\mathbb{T}^{d},$ which is not contained in any rational affine subspaces, the $\times 2, \times 3,\times 5$ orbit of $x$ is dense in $\mathbb{T}^d.$ This is not true. In fact, \cite[Theorem 2]{M10} shows that this is not true for $d\geq 6.$ The situation is unclear for $d=3,4,5.$
	\end{rem}
	We point out a direct consequence of Theorem \ref{MAIN} on simultaneous Diophantine approximation.
	\begin{cor}\label{Cor}
		Let $x,y$ be irrational numbers which are not rationally related. Then for each $(\alpha,\beta)\in [0,1)^2$ we have
		\[
		\liminf_{k_2,k_3,k_5\in \mathbb{N}} \max\{\|2^{k_2}3^{k_3}5^{k_5} x-\alpha\|,\|2^{k_2}3^{k_3}5^{k_5} y-\beta\|\}=0.
		\]
		Here $\|x\|$ is the distance between $x$ and the set $\mathbb{Z}.$
	\end{cor}
	Our method cannot provide any quantitative estimate of the above Diophantine approximation. A related problem is a conjecture of Littlewood (see more details in \cite{EKL}) says that for $x,y$ being irrational numbers,
	\[
	\liminf_{n\to\infty} n\|n x\|\|n y\|=0.
	\]
	However, if we restrict the set of choices of $n$, then it is unlikely that the above approximation can hold. In fact, it is interesting enough to see whether it is possible to show that
	\[
	\liminf_{k_2,k_3,k_5\to\infty} f(k_2,k_3,k_5) \|2^{k_2}3^{k_3}5^{k_5} x\|\|2^{k_2}3^{k_3}5^{k_5} y\|=0.
	\]
	for a suitable function $f$ such that
	\[
	\lim_{k_2,k_3,k_5\to\infty} f(k_2,k_3,k_5)=\infty,
	\]
	for example $f(k_2,k_3,k_5)= \log(k_2k_3k_5+2).$

	\section{Notations}
	\begin{itemize}
		\item
		In this paper, for integer $d\geq 1$ we identify $[0,1]^d$ with $\mathbb{T}^d=\mathbb{R}^d/\mathbb{Z}^d$ in the standard way. Unless otherwise mentioned, given $a\in \mathbb{T}^d$ and by multiplying an integer $k$ we obtain $ka\in\mathbb{T}^d$ (the $\mod 1$ action is implicitly performed). This action can be written as a diagonal matrix $diag(k,\dots,k).$ We write this action as $T_k.$
		\item
		We will use $\Pi$ to denote projections (or coordinate components). For example, let $a=(x,y)\in\mathbb{T}^2.$ We write $\Pi_1 a=x, \Pi_2 a=y.$ Sometimes, we use $(x,y)$ with negative values of $x$ or $y$. In this case we consider $(x,y)$ to be the corresponding point $(x',y')$ on $\mathbb{T}^2$ with $(x-x',y-y')\in\mathbb{Z}^2.$ More generally, for any set $E\subset\mathbb{R}^d,$ we also consider $E$ as a subset of $\mathbb{T}^d$ in a natural way. Thus a line in $\mathbb{R}^d$ is actually a union of line segments in $[0,1]^d$. 
		\item
		For $n\geq 1, a_1,\dots,a_n\in GL_{d}(\mathbb{Z})$ we write $Orb_{a_1,\dots,a_n}(x_1,\dots,x_d)$ for the orbit
		\[
		\{T^{k_1}_{a_1}\dots T^{k_n}_{a_n}(x_1,\dots,x_d)\}_{k_1,\dots,k_n\in\mathbb{N}},
		\] 
		We say that $A\subset\mathbb{T}^d$ is $T_{a_1},\dots,T_{a_n}$ invariant if for each $i\in\{1,\dots,n\},$
		\[
		T_{a_i}(A)\subset A.
		\]
		\item
		We will use the notion of directions of lines in $\mathbb{R}^d$ or $\mathbb{T}^d.$ Directions can be viewed as elements in $S^{d-1}/\{\pm 1\}$ or $\mathbb{P}^{d-1}(\mathbb{R}).$ We prefer to have $S^{d-1}$ in mind as it has a nice geometric shape. However, sometimes we are interested in rationalities of directions. In this situation, we represent a direction as an element in $\mathbb{P}^{d-1}(\mathbb{R})$ with coordinate $(1,t_1,\dots,t_{d-1})$ (it does not matter whether we normalise the first coordinate or others, since $\mathbb{P}^{d-1}(\mathbb{R})$ is covered by such charts). We say that $t$ is irrational, if $t_1,\dots,t_{d-1}$ are linearly independent over the field of rational numbers. Otherwise, we say that $t$ is rational. In some situations, it is more convenient to talk about the $\mathbb{Q}$ dimension of $Span_{\mathbb{Q}}(t_1,\dots,t_{d-1})$. We denote it as $\dim_{\mathbb{Q}}(t).$ This number does not depend on the chart we chose. (It is coordinate free.) We say that $t$ is irrational if $\dim_{\mathbb{Q}}=d-1.$ 
		\item
		Let $l$ be a line in $R^2.$ Then we can also view it on $\mathbb{T}^2.$ If the direction of $l$ is rational, then $l$ is a union of finitely many line segments in $\mathbb{T}^2.$ We say that $l$ is homogeneous if $l$ passing through the origin, i.e. $l\subset\mathbb{R}^2$ intersects $\mathbb{Z}^2.$ We say that $l$ is rational if the direction of $l$ is rational and there is an integer $N>0$ such that $T_N (l)$ is homogeneous. In other words, a rational line is a line with rational direction containing at least one (then necessarily infinitely many) rational points.
		\item
		We will be mostly considering diagonal actions on $\mathbb{T}^d.$ Let $L\in SL_{d}(\mathbb{Z}).$ Let $a\in\mathbb{T}^d$ and $k\in\mathbb{Z}.$ Then in $\mathbb{T}^d$ we have
		\[L^{-1}T_k(La)=T_k(a).\]
		Thus we can `change coordinates' by multiplying a matrix $L\in SL_d(\mathbb{Z}).$ In the new coordinate system, $T_k$ still acts as $diag(k,\dots,k).$
	\end{itemize}

	\begin{comment}
	This type coordinate change is still too restrictive. Let $L\in GL_d(\mathbb{Z})$ be a matrix. Then $\Lambda=L\mathbb{Z}^d$ is in general a sublattice of $\mathbb{Z}^d.$ We consider the enlarged turus $\tilde{\mathbb{T}}=\mathbb{R}^d/\Lambda.$ We can still consider action of $T_k=diag(k,\dots,k)$ on $\tilde{\mathbb{T}}.$ This action is well defined. Let $A\subset\mathbb{T}^d$ be a $T_k$ invariant set. We can extend $A$ periodically (modulo $\mathbb{Z}^d$) to $\mathbb{R}^d.$ We denote this extended set as $A^{\mathbb{R}}.$ Then we choose a fundamental domain $F$ of $\Lambda$ and consider $\tilde{A}=A^{\mathbb{R}}\cap F$ then consider $\tilde{A}$ as a subset of $\tilde{\mathbb{T}}.$ Then $\tilde{A}$ is a $T_k$ invariant set in $\tilde{\mathbb{T}}.$ Geometrically, $\tilde{A}$ looks like a finite union of identical copies of $A.$
	\end{comment}
	\section{$\times 2, \times 3$ phenomena and Furstenberg's argument}\label{Fu}
	Here we discuss an important observation due to Furstenberg. Let $a_n\in (0,1),n\geq 1$ be a sequence decaying to $0.$ Let $\epsilon>0$ be a small number and let $a_n$ be such that $|a_n|<\epsilon.$ We want to consider $2^{k_2}3^{k_3} a_n$ for $k_2,k_3\in\mathbb{N}.$ Without loss of generality, we can assume that $a_n>0.$ Here we view everything in $\mathbb{R}$ rather than in $\mathbb{T}.$ Taking logarithms, we want to consider
	\[
	k_2\log 2+k_3\log 3+\log a_n.
	\]
	Since $\log 2/\log 3\notin\mathbb{Q}$ we see that the additive semigroup $S_{2,3}$ generated by $\log 2, \log 3$ is eventually dense in the sense that for each $\delta>0$ there is a number $M>0$ such that 
	\[
	S_{2,3}\cap [M,\infty)
	\]
	is $\delta$ dense. Now as $a_n\to 0$ we see that $\log a_n\to -\infty.$ This implies that $\{k_2\log 2+k_3\log 3+\log a_n\}_{k_1,k_2\geq 0}\cap [-0.5,0]$ can be arbitrarily dense by letting $n$ be sufficiently large. Taking exponentials, we see that for large $n,$ $\{2^{k_2}3^{k_3}a_n\}_{k_2,k_3\geq 0}\cap [e^{-0.5},1]$ can be arbitrarily dense. Of course, the $e^{-0.5}$ here can be taken to be any positive number smaller than $1.$
	
	Using the above argument, one can show that if $0$ is not an isolated point of a closed $\times 2,\times 3$ invariant subset $A\subset\mathbb{T},$ then $A=\mathbb{T}.$ Later on, we will use this observation many times. For convenience, we call this type of argument to be a \emph{Furstenberg argument}. We mention here that we will not only apply this argument in one-dimensional situations. For such an example, see the proof of Lemma \ref{LM1}.
	\section{Local directions}
	Here we introduce the notion of local directions.
	\begin{defn}
		Let $A\subset\mathbb{T}^2$ and $a\in \mathbb{T}^2.$ Then we construct the following set
		\[
		Dir_a(A)=\left\{t\in S^1/\{\pm 1\}\approx \mathbb{P}^1(\mathbb{R}): \exists a_n\in A\setminus \{a\}, a_n\to a, \frac{a_n-a}{|a_n-a|}\to t \right\}.
		\] 
		Then let $Dir_{\mathbb{Q}}(A)$ to be the union of $Dir_a(A)$ with rational points $a\in\mathbb{T}^2.$
	\end{defn}
	Given this definition, it is very natural to call $Dir_a(A)$ to be the local direction set of $A$ around $a$. Notice that if $a$ is not contained in the closure of $A$ then $Dir_a(A)$ is empty. The usefulness of the notion of local directions can be seen in the following result.
	\begin{lma}\label{LM1}
		Let $A\subset \mathbb{T}^2$ be a closed and $T_2, T_3$ invariant set. Let $a\in A$ be a rational point. Then for each rational $t\in Dir_a(A),$ there is a rational point $a'\in A$  and the line passing through $a'$ in direction $t$ is contained in $A.$ If $Dir_a(A)$ contains an irrational direction, then $A=\mathbb{T}^2.$
	\end{lma}
	\begin{proof}
		The orbit of $a$ under $T_2$ and $T_3$ is finite. If $a$ would be fixed by $T_2, T_3$ (in this case, $a=(0,0)$) then the result follows by a direct application of Furstenberg argument and the fact that $A$ is closed. To see this, we fix a small number $\epsilon.$ By the definition of $Dir_{a}(A)$, we see that for each $t\in Dir_{a}(A)$ and for each $\delta>0,$ there is a point $a_\delta\in A\setminus\{a\}$ such that
		\[
		|a_\delta-a|<\delta
		\]
		and
		\[
		\left|\frac{a_\delta-a}{|a_\delta-a|}-t \right|<\delta.
		\]
		Here, we take $t$ as an element in $S^1$ (originally it was an element in $S^1/\{\pm 1\}$).  In other words, the distance between $a,a_\delta$ is at most $\delta$ and the orientation of the line segment $aa_\delta$ is $\delta$-close to $t.$ Now, we consider the line $l\subset\mathbb{T}^2$ passing through $aa_\delta.$ This line can be dense in $\mathbb{T}^2,$ but this does not leave us any difficulties. For example, we can just view this line simply on $\mathbb{R}^2.$ By considering $T_2,T_3$ actions only on this line, we are essentially in the situation discussed in Section \ref{Fu}. We identify $l$ with $\mathbb{R}$ in a metric preserving manner. Then $T_2,T_3$ actions simply become $\times 2, \times 3$ on $\mathbb{R}.$ As $a$ is fixed by $T_2,T_3$ we can take $a$ to be $0.$ Then $a_\delta$ is taken to be a small positive number $a'_\delta\leq \delta$. By Furstenberg argument, we see that the $\times 2 , \times 3$ orbit of $a'_\delta$ is $\delta'$ dense on $[e^{-0.5},1].$ Here $\delta'>0$ and $\delta'\to 0$ as $\delta\to 0.$ By letting $\delta$ to be small enough, we can ensure that $\delta'<\epsilon$ as well as $\delta<\epsilon.$ Then taking everything back to $\mathbb{T}^2,$ we see that for each $\epsilon>0,$ there is a line $l_\epsilon$ passing through $a$ with direction $t_\epsilon$ being $\epsilon$-close to $t$ such that the line segment $[a+e^{-0.5}t_\epsilon, a+t_\epsilon]\cap A$ is $\epsilon$-dense. By letting $\epsilon\to 0$ we see that
		\[
		[a+e^{-0.5}t, a+t]\subset A.
		\]
		In the above argument, $e^{-0.5}$ can be taken to be any positive number smaller than one. Thus we see that there is a line segment containing $a$ which is contained in $A.$ By applying $T_2,T_3$ we see that the ray passing through $a$ with direction $t$ must be contained in $A.$ If $t$ is rational, then this ray is a homogeneous line in $\mathbb{T}^2.$ If $t$ is irrational, this ray must be dense in $\mathbb{T}^2.$
		
		If $a$ is not fixed, then $Orb_{2,3}(a)$ is a finite set. Again, by Furstenberg argument we see that for each $s>0,$ there is a sequence of directions $t_n\in S^1,$ points $a_n\in A$, $b_n\in T_{2,3}(a)$ such that
		\[
		t_n=(a_n-b_n)/|a_n-b_n|\to t, |a_n-b_n|\to s.
		\]
		Since $Orb_{2,3}(a)$ is finite, we can consider $b_n$ to be all equal to $a'_s\in Orb_{2,3}(a).$ Applying the above argument for each $s>0$ we find an at most finite decomposition of $(0,\infty).$ The closure of at least one of components in this decomposition contains intervals (Baire Category). Together with the $T_2,T_3$ invariance and closeness of $A$, we conclude the result.
	\end{proof}
	To warm-up, we illustrate a simple observation.
	\begin{lma}\label{ONE}
		Let $(x,y)\in\mathbb{T}^2$ be such that $x,y$ are not rational nor rationally related. Then \[\#Dir_{\mathbb{Q}}(Orb_{2,3}(x,y))>1.\]
	\end{lma}
	\begin{rem}
		We will eventually show that under the same condition, the quantity $\#Dir_{\mathbb{Q}}(Orb_{2,3,5}(x,y))$ (with an additional $5$) cannot be finite (instead of only being bigger than $1$).
	\end{rem}
	\begin{proof}
		First, we show that $\#Dir_{\mathbb{Q}}(Orb_{2,3}(x,y))>0.$ Indeed, since $x,y$ are not rational, because of Furstenberg's theorem, by applying $\times 2, \times 3$ on $x$ one can approach any point in $[0,1].$ In particular, there is a sequence $a_n\in Orb_{2,3}(x,y)$ with $\Pi_1 a_n$ approaching a rational number. We assume that $y'=\lim_{n\to\infty}\Pi_2 a_n$ exists by passing to a subsequence if necessary. If $y'$ is irrational, then by Furstenberg's theorem and Baire category theory, the closure of $Orb_{2,3}(x,y)$ contains at least one vertical line with rational $X$-coordinate and $Dir_\mathbb{Q}(Orb_{2,3}(x,y))$ contains $(0,1).$  If $y'$ is rational, then by the compactness of $S^1$ we conclude that $\#Dir_\mathbb{Q}(Orb_{2,3}(x,y))>0.$ 
		
		Now suppose that $\#Dir_\mathbb{Q}(Orb_{2,3}(x,y))=1.$ Then this one element $t$ in $S^1$ must be rational, namely, $t=(t_x,t_y)$ satisfies $t_xt_y=0$ or $t_y/t_x\in\mathbb{Q}.$ Otherwise, $Orb_{2,3}(x,y)$ must be dense by Lemma \ref{LM1}. Assume that $t_x\neq 0, t_y/t_x=p/q$ with $gcd(p,q)=1.$ Then there is a matrix $L\in SL_2(\mathbb{Z})$ such that the direction of $L(t)$ is horizontal ($X$-coordinate direction). As $T_2,T_3$ commute with $L$ we claim that it is of no loss of generality if we assume that the direction $t$ is $(1,0).$ Indeed, for doing this, we also need to `shift' the starting point to be $(x',y')=(x,y)\times L^{T}.$ Rationality and rationally relation cannot be changed under this multiplication with matrix $L^{T}\in SL_2(\mathbb{Z}).$ More precisely, $x',y'$ are again irrational numbers which are not rationally related. Thus we may replace $x,y$ by $x',y'$ and assume that $(1,0)$ is the only element in $Dir_{\mathbb{Q}}(Orb_{2,3}(x',y')).$ Since $y'$ is irrational, we see that the closure of $Orb_{2,3}(0,y')$ contains $ \{0\}\times [0,1].$ Let $Q$ be a large prime and let $a_n\in Orb_{2,3}(x',y')$ with $\Pi_2 a_n\to Q^{-1}.$ We assume that $x''=\lim_{n\to\infty}\Pi_1 a_n$ exists. If $x''$ is irrational, then by using Furstenberg's theorem, we see that there is an integer $k$ not being a multiple of $Q$ such that $Orb_{2,3}(x',y')$ is dense $[0,1]\times \{k/Q\}.$ If $x''$ is rational, then $Dir_{(x'',y'')}(Orb_{2,3}(x',y'))=\{(1,0)\}.$ This time we apply Furstenberg argument, Lemma \ref{LM1} and see that there is an integer $k$ not being a multiple of $Q$ such that $Orb_{2,3}(x',y')$ is dense $[0,1]\times \{k/Q\}.$ As this happens for any prime number $Q>5$ we see that $\overline{Orb_{2,3}(x',y')}$ contains infinitely many horizontal lines with rational $Y$-coordinates. Those $Y$-coordinates must have at least one accumulate point. If this accumulate point is irrational, then we can apply Furstenberg's theorem. If this accumulate point is rational, then we can apply Furstenberg argument (restricted on the line $\{0\}\times [0,1]$). In both cases, we see that $\overline{Orb_{2,3}(x',y')}=\mathbb{T}^2.$ However, this certainly contradicts our assumption that
		\[
		\#Dir_{\mathbb{Q}}(Orb_{2,3}(x,y))=1.
		\]
		From here the result follows.
	\end{proof}
	Our strategy for proving Theorem \ref{MAIN} is to show that there are many directions in $Dir_\mathbb{Q}(A)$ for $A$ being a suitable $T_2,T_3,T_5$ invariant set. In order to see why this strategy is plausible, we prove the following simple observation.
	\begin{lma}\label{manydirections}
		Let $A\subset \mathbb{T}^2$ be a closed and $T_2, T_3$ invariant set. If $\#Dir_\mathbb{Q}(A)=\infty,$ then $A=\mathbb{T}^2.$
	\end{lma}
	\begin{proof}
		Since  $\#Dir_\mathbb{Q}(A)=\infty,$ we can find at least countably infinitely many elements in $Dir_\mathbb{Q}(A).$ For each $t\in Dir_\mathbb{Q}(A)$, by Lemma \ref{LM1}, we see that $A$ contains a line in direction $t.$ Without loss of generality we assume that all such $t$ must be rational directions. Let $t=(t_x,t_y)$ is such that $t_y/t_x=p/q, gcd(p,q)=1.$ Then a line with direction $t$ is $1/q$-dense in $\mathbb{T}^2.$ Thus if $A$ is not $\mathbb{T}^2,$ there are only finitely many such $q's.$ Therefore we can only find finitely many rational directions in $Dir_\mathbb{Q}(A).$ This contradiction gives us the result.
	\end{proof}
	
	\section{Pre-image tracking method and the final step}\label{track}
	So far, we have only considered $Orb_{2,3}(.)$ instead of $Orb_{2,3,5}(.).$ We shall see in this section that the addition $\times 5$ action plays a crucial role in our argument. However, we believe that the need for the additional $\times 5$ action should not be necessary, see Remark \ref{23}. In order to see clearly our reliance on the additional $\times 5$, we prove the following lemma.
	\begin{lma}\label{235}
		Let $N>1$ be an integer and let $r$ be a rational number not in $\mathbb{Z}N^{-1}.$ Then there are choices of two numbers $a,b$ in $\{2,3,5\}$ such that $a^{m}b^{n}r$ is never in $\mathbb{Z}N^{-1}$ whenever $n,m\geq 0.$ In particular, if $r\in \mathbb{T}$ is a non-zero rational number, then there exist two numbers $a,b$ in $\{2,3,5\}$ such that $a^mb^nr$ is never zero in $\mathbb{T}$, in other words, the number $a^mb^nr$ is never an integer.
	\end{lma}
	\begin{proof}
		Let $r=p/q$ with $gcd(p,q)=1.$ Suppose that there is an integer $K$ such that $Kr\in \mathbb{Z}N^{-1}$ which is equivalent to $KNr\in\mathbb{Z}.$ Then we see that
		\[
		q| KNp.
		\]
		Since $gcd(p,q)=1$ we see that
		\[
		q| KN.
		\]
		This forces $[q,N]|KN.$ Thus $K$ must be an integer multiple of $[q,N]/N.$ Now we write
		\[
		[q,N]/N=2^{k_2}3^{k_3}5^{k_5}Q
		\]
		with $k_2,k_3,k_5\geq 0, Q\geq 1, gcd(2,Q)=gcd(3,Q)=gcd(5,Q)=1.$ If $k_2>0$ then we can choose $a=3,b=5.$ If $k_3>0$, then we can choose $a=2,b=5.$ If $k_5>0$, then we can choose $a=2,b=3.$ In all these three cases, we see that
		\[
		a^{m}b^nr
		\]
		can never be in $\mathbb{Z}N^{-1}$ as $a^{m}b^{n}$ can never be an integer multiple of $[q,N]/N.$ Suppose now that $k_2=k_3=k_5=0, Q>1.$ Then we can choose $a=2,b=3.$ (In fact, any choices of $a,b$ in this case would be equally possible.) Thus we are left with the case $k_2=k_3=k_5=0, Q=1.$ In this case we see that $[q,N]=N$ and this implies that $q| N$ which is not possible because $r=p/q\notin \mathbb{Z}N^{-1}.$
	\end{proof}
	
	In what follows, we need a special construction of a set. Let $L_0$ be the unit segment connecting $(0,0)$ and $(0,1).$ Let
	\[
	a_0=(0,0), a_1=(0,N^{-1}),\dots, a_{N}=(0,1).
	\]
	Then we have the line segments $a_{i}a_{i+1}, i\in\{0,\dots,N-1\}.$ Let $0<\delta<0.0001$ be a small number. We construct a rhombus $E_0$ with vertices $(0,0)$, $(0,1)$, $(-\delta,0.5)$, $(\delta,0.5).$ Then we shrink $E_0$ with factor $N^{-1}$ and copy it along all segments $a_{i}a_{i+1}, i\in\{0,\dots,N-1\}.$ As a result we obtain a set $E_{\delta,N}$ which is a union of small rhombuses, see figure \ref{fig:figure 1}. One property we need for $E_{\delta,N}$ is as follows.
	
	\begin{itemize}
		\item{Pre-image property(PP):} If $a\notin E_{\delta,N}$ and $|\Pi_1 a|\leq \delta$ then $T_2(a)$, $T_3(a)$, $T_5(a)$ are not in $E_{\delta,N}.$
	\end{itemize}
	We claim that for each $\delta<0.0001$ and $N\geq 1$, the set $E_{\delta,N}$ has (PP). To see this,  let $a$ be a point as in the condition (PP). Then there is a point $a_i,i\in\{0,\dots,N-1\}$ such that the line segment $a_ia$ intersects $E_{\delta,N}$ only at $a_i.$ Consider the line segment $S_i$ with the centre point being $a_i$,  being parallel to $a_ia$, and whose length is $10$ times that of $a_ia$.  Since $\delta<0.0001,$ $S_i$ still intersects $E_{\delta,N}$ at $a_i$ only. Now consider the union $S=\cup_{k\in\{0,\dots,N-1\}} (S_i+k/N).$ We see that $T_2(a),T_3(a),T_5(a)$ are contained in $S.$ Since $S\cap E_{\delta,N}=\{a_0,\dots,a_N\}$ the claim follows.
	
	\begin{figure}[h]
		\includegraphics[width=0.3\linewidth, height=6cm]{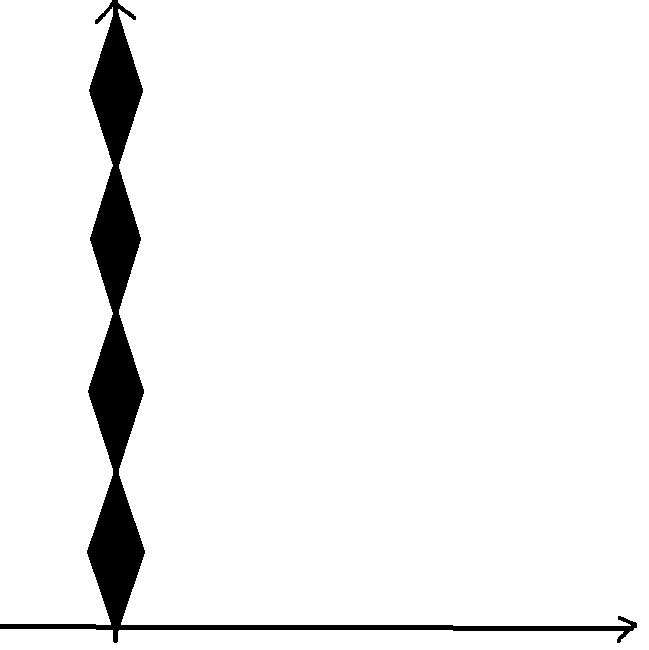} 
		\caption{Illustration of  $E_{\delta,N}$ (filled with black)}
		\label{fig:figure 1}
	\end{figure}

	For convenience, we introduce the notion of a linear form system.
	\begin{defn}
		Let $a,b\in\mathbb{Z}$ be such that at least one of them is not zero and $gcd(a,b)=1.$ Let $L_{a,b}(x,y)=ax+by$ be the corresponding linear form and $L_{a,b}$ be the line set $\{(x,y)\in\mathbb{T}^2: ax+by=0 \}.$ Let $J>0$ be an integer, let $L_0,\dots,L_J$ be different lines with $L_0=\{0\}\times [0,1].$ We call such a collection of lines a \emph{linear form system}. 
	\end{defn}
	The following result is a very important step towards the proof of Theorem \ref{MAIN}. Essentially it says that for a closed $T_2,T_3,T_5$ invariant set $A$, if $A\cap \mathbb{Q}^2$ is contained in a linear form system, then $A$ itself must be rather special.
	\begin{lma}(Pre-image tracking)\label{MAINLEMMA}
		Let $L_j,j\in \{0,\dots,J\}$ be a linear form system. Let $(x,y)\in\mathbb{T}^2$ be such that $\overline{\{Orb_{2,3,5}(x,y)\}}\cap \mathbb{Q}^2 \subset \bigcup_{0\leq j\leq  J}L_j.$ Suppose that $(0,1)\in Dir_{\mathbb{Q}}(Orb_{2,3,5}(x,y))$ then $x$ is rational.
	\end{lma}
	\begin{rem}
		The last condition is not restrictive.  In fact, as long as the direction set $Dir_\mathbb{Q}(Orb_{2,3,5}(x,y))$ contains a rational direction, by changing the coordinates, we can assume, without the loss of generality,  that this direction is $(0,1)$. This assumption makes the $X$-coordinate special as we will see in the proof.
		
		Later on, we will apply this lemma in the case when $x,y$ are not rational nor rationally related. In this case, from Lemma \ref{ONE}, we see that $Dir_{\mathbb{Q}}(Orb_{2,3,5}(x,y))$ contains at least one element. If this element is not rational, then $Orb_{2,3,5}(x,y)$ is dense and this implies that $Dir_{\mathbb{Q}}(Orb_{2,3,5}(x,y))$ contains all the directions (in particular, a rational direction).
	\end{rem}
	\begin{proof}
		Assume that $x$ is irrational. Let $S\subset\mathbb{N}$ be the multiplicative semigroup generated by $2,3,5$ and for each natural number $n$, $S_n$ be the $n$-th element in $S.$ For convenience let $a_0=(x,y)$, $a_n=(S_n x, S_n y)\in\mathbb{T}^2$ and $A=\overline{\{Orb_{2,3,5}(x,y)\}}.$ Our assumption is that for each subsequence $\{a_{n_k}\}_{k\geq 1},$ whenever \[\lim_{k\to\infty}\Pi_1 a_{n_k}\in \mathbb{Q}\setminus \{0\},\] we have (here $\lim$ denotes the set of limit points) \[\lim_{k\to\infty} a_{n_k}\subset\bigcup_{J\geq j>0} L_j.\] 
		This is because if $\lim_{k\to\infty} \Pi_2 a_{n_k}$ contains an irrational number, then by choosing carefully two actions in $\{T_2,T_3,T_5\}$ (Lemma \ref{235}) and applying Furstenberg's theorem we can find a vertical line with rational $X$-coordinate not passing through the origin. This line certainly contains rational points not on $\bigcup_{j\in J} L_j.$ From here we see that $\{0\}\times [0,1]$ is the only vertical line with rational $X$ coordinate that could be contained in $A.$ Since $(0,1)\in Dir_\mathbb{Q}(A),$ by Lemma \ref{LM1}, we see that $A$ contains a vertical line with a rational $X$-component.   We conclude that $\{0\}\times [0,1]$ is the only vertical line with rational coordinate that is contained in $A.$
		
		Let $a\in Orb_{2,3,5}(x,y)$ then there are uniquely determined (recall that $x$ is irrational by assumption) non-negative integers $k_2,k_3,k_5$ with
		\[
		a=T^{k_2}_2T^{k_3}_3T^{k_5}_5(x,y).
		\]
		We need to determine a `pre-image' of $a$. Since at least one of $k_2,k_3,k_5$ is strictly positive, say $k_2>0$, we can choose $pre(a)$ to be
		\[
		T^{k_2-1}_2T^{k_3}_3T^{k_5}_5(x,y).
		\]
		If multiple choices are possible, we have a free choice. For concreteness, we prefer to minus one on $k_2$ than on $k_3$ than on $k_5.$ In this way, we defined $pre(a)$ for all $a\in Orb_{2,3,5}(x,y)$ expect for $(x,y).$ We will not need to apply $pre$ on $(x,y)$ so we leave the function $pre$ as it is without trying to extend the domain.

		Let $N$ be an integer such that $(0,y')\in \bigcup_{j>0} L_j\implies y'N\in\mathbb{Z}.$ We cut the line $L_0=\{0\}\times [0,1]$ into equal pieces with length $N^{-1}$. The cut points are rational numbers of form $m/N, m\in\{0,\dots,N\}.$ We choose a small number $\delta>0$ and construct $E_{\delta,N}$ as mentioned before this lemma. For simplicity, we call $E=E_{\delta,N}.$ 
		
		Now we want to consider sequences $a_{n_k},k\geq 1$ with $a_{n_k}\in E$ and \[\lim_{k\to\infty} {\Pi_1 a_{n_k}}=0.\] We need to be sure that there is at least one such a sequence. Assume that there is none then we see that $A\cap L_0\subset \{m/N\}_{m\in\{0,\dots,N\}}.$ This contradicts the fact that $L_0\subset A.$ Having shown the existence, we take an arbitrary such sequence $\{a_{n_k}\}_{k\geq 1}.$ Consider the pre-images
		\[
		\{pre(a_{n_k})\}_{k\geq 1}.
		\]
		For each $k$, we see that there is an action $T$ in $\{T_2,T_3,T_5\}$ such that
		\[
		T(pre(a_{n_k}))=a_{n_k}.
		\]
		Now  $pre(a_{n_k})$ can be still near the line $L_0$ (in $\mathbb{T}^2$) in which case $\Pi_1 pre(a_{n_k})=(\Pi_1 a_{n_k})/s$ or $1-(\Pi_1 a_{n_k})/s$ for a number $s\in \{2,3,5\}$. Otherwise, $pre(a_{n_k})$ is far away from $L_0.$ More precisely, we see that for a number $s\in \{2,3,5\}$ 
		\[
		s(\Pi_1 pre(a_{n_k}))=\Pi_1 a_{n_k}.
		\]
		If $s=2$ then $\Pi_1 pre(a_{n_k})$ would be close to $1/2.$ (Similar argument works for $s=3,5$ as well.) Suppose that we can find infinitely many $k$ such that $pre(a_{n_k})$ is away from $L_0.$ By passing to a subsequence we see that there is a 
		\[
		t\in \{1/5,2/5,3/5,4/5,1/3,2/3,1/2\}
		\]
		and
		\[
		\lim_{k\to\infty} \Pi_1 pre(a_{n_k})=t.
		\]
		By passing to a further subsequence if necessary, we assume that the limits
		\[
		\lim_{k\to\infty} \Pi_2 a_{n_k}=y'',\lim_{k\to\infty} \Pi_2 pre(a_{n_k})=y'''
		\]
		exist. 
		
		If $y''$ is irrational, then $y'''$ is irrational as well. Then we see that $(t,y''')\in A.$ The denominator of $t$ is in $\{2,3,5\}.$ By Lemma \ref{235}, we can choose two numbers $a,b$ out of $\{2,3,5\}$ and by applying Furstenberg's theorem for actions $T_a,T_b$ we see that there is a $r\neq 0$ and $\{r\}\times [0,1]$ is contained in $A.$ This contradicts the assumption. 
		
		Now, suppose that $y''$ is rational. Then $y'''$ is rational as well. If $(t,y''')\notin \bigcup_{j>0} L_j$ then we would have a contradiction. If $(t,y''')\in \bigcup_{j>0} L_j,$ then $y''$ must be of form $m/N, m\in\{0,\dots,N\}.$ This is because $T_2,T_3,T_5$ preserve $L_j,j>0.$ If we choose $\delta$ to be small enough, then by the construction of the set $E$ we see that $(0,y'')$ cannot be approached by the sequence $a_{n_k}$ within directions determined by $L_j,j>0.$ Since $T_2,T_3,T_5$ preserve directions, we see that $(t,y''')$ cannot be approached by $pre(a_{n_k})$ within directions determined by $L_j,j>0.$ 
		
		Let $v$ be a direction of which $(t,y''')$ is approached, namely, \[v\in Dir_{(t,y''')}(Orb_{2,3,5}(x,y)).\] We assume that $v$ is rational. Otherwise, we have $A=\mathbb{T}^2.$ Using Lemma \ref{235}, we can choose a suitable set of two actions in $\{T_2,T_3,T_5\}$ depending on the denominator of $t$ and we see that there is a rational point $(t',y'''')\in A$ with $t'\neq 0$ such that the line passing through $(t',y'''')$ with direction $v$ is contained in $A.$ (Here it is important that $t'\neq 0$ as $v$ can be $(0,1),$ the vertical direction. In this case the following argument would fail. This is where the additional $\times 5$ action plays a crucial role.) We know that $v$ is different than any of the directions of $L_j,j>0.$ Then we see that $A$ must contain rational points outside the $Y$-axis which are not on $\bigcup_{j>0} L_j.$ This contradicts the assumption.
		
		Denote $d(.,.)$ to be the Euclidean metric of $\mathbb{T}^2.$	The above contradiction shows that all but only finitely many $pre(a_{n_k})$ can be away from $L_0.$ As this holds for any sequence $a_{n_k}$ approaching $L_0$ within $E$, we see that there is a small number $\epsilon>0$ such that whenever $a\in E\cap Orb_{2,3,5}(x,y), d(a,L_0)<\epsilon$, the pre-image $a'$ of $a$ (as $a$ is in the orbit, this pre-image is uniquely determined in our construction) is closer to $L_0.$ More precisely, we have
		\[
		d(a',L_0)\leq 0.5d(a,L_0).
		\] 
		By the \emph{pre-image property (PP)}, we see that $a'\in E$ as well (recall that we chose $d(a,L_0)$ to be very small). Thus we can track the pre-image of $a'$, then the pre-image of the pre-image... After a certain number $M$ of tracking back steps, we must arrive at $(x,y)$, the starting point. We see that $d((x,y),L_0)\leq 2^{-M} \epsilon.$ Since $M$ can be arbitrarily large and $\epsilon$ can be arbitrarily small, we see that $x=0$ (or $1$, which is the same because it is in $\mathbb{T}$). This contradiction shows that the assumption that $x$ is irrational cannot hold. Thus we see that $x$ is rational.
	\end{proof}
	
	We will now finish the main step of the proof. In order to give a clear illustration, we prove the following lemma which will be upgraded into a full proof of Theorem \ref{MAIN}. This lemma is only a tiny improvement of Lemma \ref{ONE}, but the proof is much more complicated.
	\begin{lma}\label{ATLEASTTWO}
		Let $(x,y)\in\mathbb{T}^2$ be such that $x,y$ are not rational nor rationally related. Then \[\#Dir_{\mathbb{Q}}(Orb_{2,3,5}(x,y))>2.\]
	\end{lma}
	\begin{rem}
		We can replace the number $2$ on the RHS with any positive integer. This will require some technical considerations which will be discussed later. 
	\end{rem}
	\begin{proof}
		Assume that
		\[
		\#Dir_{\mathbb{Q}}(Orb_{2,3,5}(x,y))=2,
		\]
		then let $\{u,v\}=Dir_{\mathbb{Q}}(Orb_{2,3,5}(x,y)).$ Both $u,v$ can be assumed to be rational direction. By applying a matrix $L$ in $SL_2(\mathbb{Z})$ (if necessary) we shall assume that
		\[
		u= (0,1), v\neq (0,1).
		\]
		By doing this we shifted our starting point $(x,y)$ to $(x',y')$ obtained by $(x,y)L^{T}.$ If the orbit closure $A=\overline{Orb_{2,3,5}(x',y')}$ contains infinitely many $u$ directional lines then $A$ must be $\mathbb{T}^2.$ In addition, all the $u$ directional lines contained in the orbit closure must be rational lines. This can be seen via the $X$-coordinates of those lines after applying Furstenberg's theorem. The same holds for $v$-directional lines as well. Therefore we assume that $A$ contains finitely many $u$ and $v$ directional lines which are necessarily rational. Denote the collection of those lines as $\mathcal{L}.$ Let $N$ be an integer such that $T_N$ maps lines in $\mathcal{L}$ to homogeneous lines (passing through $(0,0)$). 	There are only two such homogeneous lines $l_u,l_v$ with directions $u,v$ respectively. This is possible because $\mathcal{L}$ is a finite collection of rational lines. More precisely, let $N'$ be such that the lines in $\mathcal{L}$ with $v$ direction intersect the $Y$-axis with coordinate in $N^{-1}\mathbb{Z}.$ This is possible since $v\neq (0,1)$ and there are finitely many lines in $\mathcal{L}$ (therefore there are finitely many intersections of lines in $\mathcal{L}$ with the $Y$-axis). Similar argument holds for $u$ directional lines in $\mathcal{L}$ with an integer $N''.$ We can now choose $N=N'N''.$ We enlarge $\mathcal{L}$ into $\tilde{\mathcal{L}}$ by including all pre-images of $l_u,l_v$ under $T_N.$
		
		Let $r$ be a rational number not of form $m/N,m\in\{0,\dots,N\}.$ We have $\{r\}\times [0,1]\cap A\neq\emptyset.$  If $(r,s)\in A$ with $s$ being irrational, then by choosing two of $\{T_2,T_3,T_5\}$ carefully as indicated in Lemma \ref{235}, we see that there is a non-zero rational number $r'$ not of form $m/N, m\in\{0,\dots,N\}$ such that $\{r'\}\times [0,1]\subset A$. This contradicts our assumption on $N.$ Thus $s$ must be rational. Then $(r,s)$ must be non-isolated. Consider $Dir_{(r,s)}(A),$ we see that 
		\[
		Dir_{(r,s)}(A)=\{v\}.
		\]
		For otherwise we could obtain an additional $u$ directional line that contradicts our assumption on $N.$ Similarly, with the help of Lemma \ref{235} again, we see that $(r,s)$ is contained in $v$ directional lines in $\tilde{\mathcal{L}}.$ Indeed, suppose that the homogeneous $v$ directional line satisfies $\{k_1 x+y=0\}$ where $k_1$ is rational. Then the $v$ directional lines in $\tilde{\mathcal{L}}$ satisfy
		\[
		\{k_1 x+y=k/N\}, k\in\{0,\dots,N-1\}.
		\] 
		If $(r,s)$ is not contained in those lines, we see that
		\[
		k_1 r+s\notin \mathbb{Z}N^{-1}.
		\]
		Thus by Lemma \ref{235}, we can choose two numbers $a,b$ out of $\{2,3,5\}$ such that the $T_a,T_b$ orbit of $(r,s)$ never intersects  the $v$ directional lines in $\tilde{\mathcal{L}}.$ Then from the proof of Lemma \ref{LM1} we can find a $v$ directional line contained in $A$ which is not one of those lines in $\tilde{\mathcal{L}}.$ This contradicts the construction of $\tilde{\mathcal{L}}.$
		
		To summarise, if $(r,s)\in A$ is rational, then it is contained in lines in $\tilde{\mathcal{L}}.$ Now we replace $x'$ by $x''=Nx'$, $y'$ by $y''=Ny'$ and consider $A''=\overline{Orb_{2,3,5}(x'',y'')}.$ It  can be checked that $Dir_\mathbb{Q}(A'')=\{u,v\}.$ We see that $A''\cap \mathbb{Q}^2$ is rather restrictive. More precisely, if $(r,s)\in A''$ is a rational point, then $A$ contains one of its pre-image under $T_N.$ This pre-image must be contained in lines in $\tilde{\mathcal{L}}$ and this implies that $(r,s)$ must be contained in $l_u$ or $l_v.$ By Lemma \ref{MAINLEMMA} we see that $x''\in\mathbb{Q}.$ However, we have $x''=Nx'$ and $x'$ is of form $ax+by$ with $a,b\in\mathbb{Z}.$ This shows that $x,y$ are rationally related. This contradicts the assumption and the result follows.
		
	\end{proof}
	
	We now finish the last step of the proof.
	
	\begin{lma}\label{manydirections2}
		Let $(x,y)\in\mathbb{T}^2$ be such that $x,y$ are not rational nor rationally related. Then \[\#Dir_{\mathbb{Q}}(Orb_{2,3,5}(x,y))=\infty.\]
	\end{lma}
	\begin{proof}
		Suppose that $k\geq 3$ is an integer and $\#Dir_{\mathbb{Q}}(Orb_{2,3,5}(x,y))=k.$ Then the elements in $Dir_{\mathbb{Q}}(Orb_{2,3,5}(x,y))$ must be rational. By a multiplication with a matrix in $SL_2(\mathbb{Z})$ we shall assume that $(0,1)$ is in this local direction set. The directions in $Dir_{\mathbb{Q}}(Orb_{2,3,5}(x,y))$ determine a linear form system $L$. For each line $l$ in this system, we see that $A=\overline{Orb_{2,3,5}(x,y)}$ can contain at most finitely many necessarily rational lines  in the direction of $l.$ We collect this lines in all possible directions as $\mathcal{L}.$	Thus we can find an integer $M$ such that $T_M$ sends lines in $\mathcal{L}$ to homogeneous lines. We consider $Orb_{2,3,5}(Mx,My).$ Let $\tilde{\mathcal{L}}$ be an extension of $\mathcal{L}$ by including all pre-images of $L$ under $T_M.$ Whenever there are points in $Orb_{2,3,5}(Mx,My)$ approaching a rational point not on the linear form system $L$, we see that there is a sequence in $Orb_{2,3,5}(x,y)$ approaching a rational number which is not contained in any of the lines in $\tilde{\mathcal{L}}.$ By Lemma \ref{235} we see that $A$ contains a rational line not in $\tilde{\mathcal{L}}$, this contradicts the construction of $\mathcal{L}$ and $\tilde{\mathcal{L}}.$ Thus we see that the only rational points on the closure of $Orb_{2,3,5}(Mx,My)$ are contained in lines in  $\tilde{\mathcal{L}}.$ By Lemma \ref{MAINLEMMA}, we see that $Mx$ must be a rational number and this would be a contradiction. From here we see that
		\[
		\#Dir_{\mathbb{Q}}(Orb_{2,3,5}(x,y))=k
		\]
		cannot hold for any integer $k\geq 3.$ As we have seen that 
		\[
		\#Dir_{\mathbb{Q}}(Orb_{2,3,5}(x,y))\geq 3
		\]
		we conclude that
		\[
		\#Dir_{\mathbb{Q}}(Orb_{2,3,5}(x,y))=\infty.
		\]
	\end{proof}
	
	Theorem \ref{MAIN} is then proved. We give a formal conclusion below.
	\begin{proof}[Proof of Theorem \ref{MAIN}]
		We see that Theorem \ref{MAIN} is a direct consequence of Lemma \ref{manydirections2} and Lemma \ref{manydirections}.
	\end{proof}

	\section{Acknowledgement}
	HY was financially supported by the University of Cambridge and the Corpus Christi College, Cambridge. HY has received funding from the European Research Council (ERC) under the European Union’s Horizon 2020 research and innovation programme (grant agreement No. 803711). HY thanks P\'{e}ter Varj\'{u} for various helpful comments.
	\providecommand{\bysame}{\leavevmode\hbox to3em{\hrulefill}\thinspace}
	\providecommand{\MR}{\relax\ifhmode\unskip\space\fi MR }
	\providecommand{\MRhref}[2]{%
		\href{http://www.ams.org/mathscinet-getitem?mr=#1}{#2}
	}
	\providecommand{\href}[2]{#2}

\end{document}